\documentclass{amsart}
\usepackage{amsmath,amssymb,amsfonts,enumerate,listings,bold-extra}
\usepackage{epsfig}
\usepackage{mathrsfs}
\usepackage{graphicx,color}
\def\N{\mathbb{N}}
\def\R{\mathbb{R}}

\newcommand{\ex}[1]{\mathsf{E}\left[\,#1\,\right]}
\newcommand{\abs}[1]{\left\vert#1\right\vert}
\newcommand{\set}[1]{\left\{#1\right\}}

\newcommand{\ind}[1]{\mathbf{1}_{#1}}

\newtheorem{theorem}{Theorem}[section]
\newtheorem{corollary}[theorem]{Corollary}

\theoremstyle{definition}
\newtheorem{definition}[theorem]{Definition}
\theoremstyle{remark}
\newtheorem{remark}[theorem]{Remark}

\begin{document}
\title{Fractional Brownian motion in a nutshell}
\author{Georgiy Shevchenko}

\address{Department of Mechanics and Mathematics, Taras Shevchenko National University of Kyiv\\
Volodymirska 60, 01601 Kyiv, Ukraine\\
E-mail: zhora@univ.kiev.ua}

\begin{abstract}
This is an extended version of the lecture notes to a mini-course devoted to fractional Brownian motion and delivered to the participants of 7th Jagna International Workshop.

\keywords{fractional Brownian motion, Hurst parameter, H\"older continuity, consistent estimation, simulation\\
2010 AMS subject classification:  60G22, 60-01, 62M09, 65C20}    
\end{abstract}

\maketitle

\section{Introduction}

The fractional Brownian motion (fBm) is a popular model for both short-range dependent and long-range dependent phenomena in various fields, including physics, biology, hydrology, network research, financial mathematics etc. There are many good sources devoted to the fBm, I will cite only few of them. For a good introductory text on the fBm, a reader may address recent Ivan Nourdin's lecture notes \cite{nourdin} or the dedicated chapter of the famous David Nualart's book \cite{nualart}. More comprehensive guides are by Yuliya Mishura \cite{mishura} and Francesca Biagini et al \cite{biagini}; the former has stronger emphasis towards the pathwise integration, while the latter, towards the white noise approach. A review of Jean-Fran\c cois Coeurjolly\cite{coeurjolly} is an extensive guide to the use of statistical methods and simulation procedures for the fBm.

It is worth saying few words on the aim and the origin of this article. After I gave a mini-course devoted to the fBm at the 7th Jagna International Conference, the organizers approached me with a proposition to write lecture notes. Knowing that there are already so many sources devoted to the fBm, I was hesitant for the first time. But ultimately I decided to agree and wrote this article. Naturally, it would be impossible to cover all the aspects of the fBm in such a short exposition, and this was not my aim. My aim was  rather to make a brief introduction to the fBm. Since most of the listeners of the course were not pure mathematicians, I tried to keep the text as accessible as possible, at the same time paying more attention at such practical issues as the simulation and identification of fBm. 

The article structured as follows. In Section~\ref{sec:def}, the fractional Brownian motion is defined, and its essential properties are studied. Section~\ref{sec:cont} is devoted to the continuity of fBm. In Section~\ref{sec:repres}, several integral representations of fBm in terms of standard Wiener process are given. Section~\ref{sec:stat} discusses the statistical estimation issues for fBm. In Section~\ref{sec:sim}, a simulation algorithm for fBm is presented.

\section{Definition and basic properties}\label{sec:def}

\begin{definition}
A \textit{fractional Brownian motion} (fBm) is a centered Gaussian process $\set{B_t^H,t\ge 0}$ with the covariance function 
\begin{equation}\label{eq:fbmcov}
\ex{B_t^H B_s^H} = \frac12 \left(t^{2H} + s^{2H} - \abs{t-s}^{2H}\right).
\end{equation}
This process has a parameter $H\in(0,1)$, called the \textit{Hurst parameter} or the \textit{Hurst index}.
\end{definition}

\begin{remark}
In order to specify the distribution of a Gaussian process, it is enough to specify its mean and covariance function, therefore, for each fixed value of the Hurst parameter $H$, the distribution of $B^H$ is uniquely determined by the above definition. However, this definition does not guarantee the existence of fBm; to show that the fBm exists, one needs e.g.\ to check that the covariance function is non-negative definite. We will show the existence later, in Section~\ref{sec:repres}, giving an explicit construction of fBm.
\end{remark}

Observe that for $H=1/2$, the covariance function is $\ex{B^{1/2}_t B^{1/2}_s} = t\wedge s$, i.e.\ $B^{1/2} = W$, a standard Wiener process, or a Brownian motion. This justifies the name ``fractional Brownian motion'': $B^H$ is a generalization of Brownian motion obtained by allowing the Hurst parameter to differ from $1/2$. Later we will uncover  the meaning of the Hurst parameter.

Further we study several properties which can be deduced immediately from the definition. The following representation for the covariance of increments of fBm is easily obtained from \eqref{eq:fbmcov}:
\begin{equation}\label{eq:fbminccov}
\begin{gathered}
\ex{\left(B_{t_1}^H - B_{s_1}^H\right)\left(B_{t_2}^H - B_{s_2}^H\right)}\\ = \frac{1}{2}\left(\abs{t_1-s_2}^{2H} + \abs{t_2-s_1}^{2H} -\abs{t_2-t_1}^{2H} - -\abs{s_2-s_1}^{2H}\right).
\end{gathered}
\end{equation}

\textbf{Stationary increments.} Take a fixed $t\ge 0$ and consider the process $Y_t = B^H_{t+s} - B^H_s$, $t \ge 0$. It follows from \eqref{eq:fbminccov} that the covariance function of $Y$ is the same as that of $B^H$. Since the both processes are centered Gaussian, the equality of covariance functions implies means that $Y$ has the same distribution as $B^H$. Thus, the incremental behavior of $B^H$ at any point in the future is the same, for this reason $B^H$ is said to have stationary increments. Processes with stationary increments  are good for modeling a time-homogeneous evolution of system.

\textbf{Self-similarity.} Now consider, for a fixed $a>0$, the process $Z_t = B^H_{at}$, $t\ge 0$. It is clearly seen from \eqref{eq:fbmcov} that $Z$ has the same covariance, consequently,  the same distribution as $a^H B^H$. This property is called $H$-self-similarity. It  means the scale-invariance of the process: in each time interval the behavior is the same, if we choose the space scale properly. 

It is an easy exercise to show that the fBm with Hurst parameter $H$ is, up to a constant, the only $H$-self-similar Gaussian process with stationary increments. 

\textbf{Dependence of increments.} Let us return to the formula \eqref{eq:fbminccov} and study it in more detail. Assume that $s_1<t_1<s_2<t_2$ so that the intervals $[s_1,t_1]$ and $[s_2,t_2]$ do not intersect. Then the left-hand side of \eqref{eq:fbminccov} can be expressed as 
$\big((f(a_1) - f(a_2) - (f(b_1)  - f(b_2)\big)/2$, where $a_1 = t_2-s_1$, $a_2 = t_2 - t_1$, $b_1 = s_2 - s_1$, $b_2 = s_2 - t_1$, $f(x) = x^{2H}$. Obviously, $a_1-a_2 = b_2 - b_1 = t_1-s_1$. Therefore, $$ \ex{\left(B_{t_1}^H - B_{s_1}^H\right)\left(B_{t_2}^H - B_{s_2}^H\right)}<0 \quad\text{for } H\in(0,1/2)$$ in view of the concavity of $f$; $$\ex{\left(B_{t_1}^H - B_{s_1}^H\right)\left(B_{t_2}^H - B_{s_2}^H\right)}>0\quad\text{for } H\in(1/2,1),$$ since $f$ is convex in this case. Thus, for $H\in(0,1/2)$, the fBm has the property of counterpersistence: if it was increasing in the past, it is more likely to decrease in the future, and vice versa. In contrast, for $H\in(1/2,1)$, the fBm is persistent, it is more likely to keep trend than to break it. Moreover, for such $H$, the fBm has the property of long memory (long-range dependence). 

Finally we mention that the fBm is neither a Markov process nor a semimartingale. 

\section{Continuity of fractional Brownian motion}\label{sec:cont}

There are several ways to establish the continuity of fBm. 
All of them are based on the formula
\begin{equation}\label{fbm-vario}
\ex{\left(B^H_t - B^H_s\right)^2} = \abs{t-s}^{2H}
\end{equation}
for the variogram of fBm, which follows from \eqref{eq:fbminccov}.

The first of the methods is probably the most popular way to prove that a process is continuous.
\begin{theorem}[Kolmogorov--Chentsov continuity theorem]
Assume that for a stochastic process $\set{X_t,t\ge 0}$ there exist such $K>0,p>0,\beta>0$ such that for all $t\ge 0, s\ge 0$
$$
\ex{\abs{X_t - X_s}^p}\le K \abs{t-s}^{1+\beta}.
$$
Then the process $X$ has a continuous modification, i.e.\ a process $\set{\widetilde{X}_t,t\ge 0}$ such that $\widetilde X\in C[0,\infty)$ and for all $t\ge 0$ \ $\Pr(X_t = \widetilde X_t)=1$. Moreover, for any $\gamma\in (0,\beta/p)$ and $T>0$ the process $\widetilde{X}$ is $\gamma$-H\"older continuous on $[0,T]$, i.e.
$$
\sup_{0\le s< t \le T} \frac{\abs{X_t-X_s}}{(t-s)^\gamma}<\infty.
$$
\end{theorem}

\begin{corollary}\label{thm:fbm-cont}
The fractional Brownian motion $B^H$ has continuous modification. Moreover, for any $\gamma\in(0,H)$ this modification is $\gamma$-H\"older continuous on each finite interval.
\end{corollary}
\begin{proof}
Since $B_t^H - B_s^H$ is centered Gaussian with variance $\abs{t-s}^H$, we have 
$\ex{\abs{B_t^H - B_s^H}^p} = K_p \abs{t-s}^{pH}$. Therefore, taking any $p>1/H$, we get the existence of continuous modification. We also get the H\"older continuity of the modification with exponent $\gamma\in(0,H-1/p)$. Choosing $p$ sufficiently large, we arrive at the desired statement. 
\end{proof}
To avoid speaking about a continuous modification each time, in the rest of this article we will assume the continuity of fBm itself. 

Another way to argue the H\"older continuity lies through a very powerful deterministic inequality. 
\begin{theorem}[Garsia--Rodemich--Rumsey inequality] 
For any $p>0$ and $\theta>1/p$ there exists a constant $K_{p,\theta}$ such that for any  $f\in C[0,T]$
\begin{equation*}
\sup_{0\le s<t\le T} \frac{\abs{f(t)-f(s)}}{(t-s)^{\theta-1/p}}\le C_{p,\theta}\left(\int_0^T \int_0^T\frac{\abs{f(x)-f(y)}^{p}}
{\abs{x-y}^{\theta p+1}}\,dx\,dy\right)^{1/p}.
\end{equation*}
\end{theorem}
\begin{remark}
One of the most widely used techniques in calculus is the estimation of integral by the supremum of integrand times measure of integration set, e.g.\ $\abs{\int_a^b f(x) dx}\le \sup_{x\in[a,b]} \abs{f(x)} (b-a)$. However, obviously, one cannot reverse this inequality and estimate the integrand by the value of integral (although the temptation is great sometimes). Now we see why the Garsia--Rodemich--Rumsey (GRR) inequality is a very striking fact (at least at first glance): it is a valid example of such reverse statement.
\end{remark}
The continuity assumption in the GRR inequality is essential. It is easy to see that for $f=\ind{[0,T/2]}$ the right-hand side of the inequality is finite, while the left-hand side is infinite. So in order to show the H\"older continuity of fBm using the GRR inequality, we should first establish usual continuity with the help of some other methods (and we have already done that). The advantage of the GRR inequality is that in contast to the Kolmogorov--Chentsov theorem it allows to estimate the H\"older norm of a process.
\begin{proof}[Alternative proof ot the second part of Corollary~\ref{thm:fbm-cont}]
We remind that we assume $B^H$ itself to be continuous.
Take some $\theta<H$ and $p>1/H$ and write, as before,
$\ex{\abs{B_t^H - B_s^H}^p} = K_p \abs{t-s}^{pH}$. Denote $$\zeta = \sup_{0\le s<t\le T} \frac{\abs{B_t^H-B_s^H}}{(t-s)^{\theta-1/p}}.$$
 Raising the GRR inequality to the power $p$ and taking expectations, we get 
\begin{align*}
\ex{\zeta^p}&\le K_{p,\theta}^p\int_0^T \int_0^T\frac{\ex{\abs{B_x^H-B^H_y}^{p}}}
{\abs{x-y}^{\theta p+1}}\,dx\,dy\\&= K_{p,\theta}^p K_p \int_0^T \int_0^T\abs{x-y}^{p(H-\theta)-1}\,dx\,dy<\infty.
\end{align*}
It follows that $\zeta <\infty$ a.s. By changing, if necessary, the fBm $B^H$ on an event of zero probability, we get the desired H\"older continuity.
\end{proof}

Finally, we mention that by using specialized facts about regularity of Gaussian processes, it is possible to show that the exact modulus of continuity of fBm is $\omega(\delta) = \delta^H \abs{\log \delta}^{1/2}$. Consequently, it is only H\"older continuous of order up to $H$, but not $H$-H\"older continuous (although quite close to be). 

Let us now summarize what we know about the Hurst parameter $H$. We already knew that, depending on whether $H\in(0,1/2)$ or $H\in(1/2,1)$, the increments of fBm are either negatively correlated or positively correlated. It is also easy to see that the correlation increases with $H$. In other words, the fBm becomes more and more persistent when $H$ increases (ultimately for $H=1$ it becomes a linear function: $B^1_t = \xi t$, where $\xi$ is standard Gaussian).

On the other hand, it follows from the above discussion that the Hurst parameter $H$ dictates the regularity of fBm: the larger $H$ is, the smoother fBm becomes. Now it is probably the most suitable moment to give some pictures of fBm, which illustrate perfectly the dependence of fBm on $H$.

\begin{figure}[htbp]
\begin{center}
\includegraphics[width=.8\textwidth]{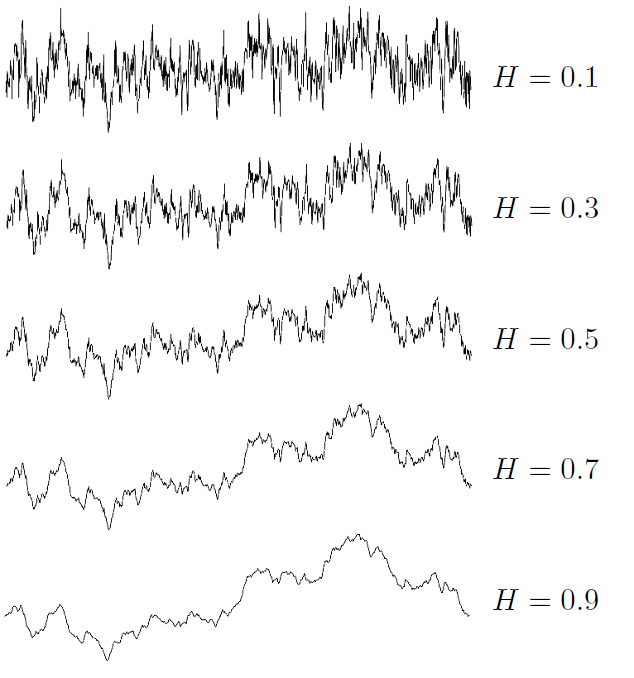}
\caption{Paths of fBm for different values of $H$.}
\end{center}
\end{figure}

\section{Integral representations of fractional Brownian motion}\label{sec:repres}

Further we will study representations of fractional Brownian motion in terms of a standard Wiener process. I expect the reader to be aware of It\^o stochastic calculus, nevertheless, it is worth to give concise information on the objects we need.

Let $\set{W_t,t\ge \R}$ be a standard Wiener process on $\R$, i.e. $\set{W_t,t\ge 0}$ and $\set{W_{-t},t\ge 0}$ are independent standard Wiener processes on $[0,\infty)$. 

For functions $f\in L^2(\R)$ the integral $I(f) = \int_{\R} f(x) dW(x)$
with respect to $W$ (the Wiener integral) is constructed as follows. For a step function 
$$
h(x) = \sum_{k=1}^n a_k \ind{[s_k,t_k]}(x),
$$
define 
$$
I(h) = \int_{\R} h(x)dW(x) = \sum_{k=1}^{n}a_k \left(W_{t_k}-W_{s_k}\right).
$$
It is easily checked that $I$ is linear and isometric, consequently, it can be extended from the set of step functions to $L^2(\R)$. This extension, naturally, is an isometry too. We summarize below its basic properties. 

\begin{enumerate}[1.]
\item linearity: for $\alpha,\beta\in\R$, $f,g\in L^2(\R)$ 
$$
I(\alpha f + \beta g) = \alpha I(f) + \beta I(g);
$$
\item mean zero: $\ex{I(f)} = 0$;
\item isometry: $\ex{I(f)^2} = \int_\R f(x)^2 dx$, moreover, for $f,g\in L^2 (\R)$
$$
\ex{I(f)I(g)} = \int_\R f(x)g(x) dx.
$$
\item for $f_1,\dots, f_n \in L^2(\R)$ the random variables $I(f_1),\dots,I(f_n)$ are jointly Gaussian.
\end{enumerate}

Next we consider representations of the form 
$$
B^H_t = I(k_t) = \int_{\R} k_t (x)dW(x),
$$
where for each $t\ge 0$ \ $k_t\in L^2(\R)$ is some deterministic kernel (not necessarily supported by the whole real line). Due to the properties of Wiener integral, the process given by such representation is a centered Gaussian process. So in order to argue that such representation defines an fBm, it is enough to show that it has the same covariance. The following simple statement may also be of use: a process has covariance given by \eqref{eq:fbmcov} iff its variogram is given by  \eqref{fbm-vario}. 

The \textit{Mandelbrot--van Ness} representation, or the moving average representation of fBm is defined in the following proposition. It also can be used as a proof of existence of fBm.

\begin{theorem}
Let for $H\in(0,1)$ $$k_t^{MA}(x) = K^{MA}_H\left((t-x)_+^{H-1/2}\ind{(-\infty,0)}(x) - (-x)_+^{H-1/2}\right), $$
where 
\begin{align*}
K^{MA}_H & = \left(\frac{1}{2H} + \int_0^\infty \big((x+1)^{H-1/2} - x^{H-1/2}\big)^2 dx \right)^{-1/2}\\ & = \frac{\left(\Gamma(2H+1)\sin \pi H\right)^{1/2}}{\Gamma(H+1/2)}.
\end{align*}
Then the process $X_t = I(k^{MA}_t)$ is an fBm with Hurst parameter $H$. 
\end{theorem}
\begin{proof}
As it was already mentioned above, in order to prove the statement, it suffices to show that for any $t,s\ge 0$ \ $\ex{(X_t-X_s)^2} = \abs{t-s}^{2H}$. 

Write, denoting $\mu = H-1/2$,
\begin{align*}
&\ex{(X_t-X_s)^2} = (K_H^{MA})^{2}\int_{\R} \big( (t-x)_+^{\mu} - (s-x)_+^{\mu}\big)^2 dx\\
&\  = (K_H^{MA})^{2}(t-s)^{2H} \int_{\R}\big( (x+1)_+^{\mu} - (x)_+^{\mu}\big)^2dx \\ 
&\ = (K_H^{MA})^{2}(t-s)^2H \left(\int_{-1}^0 (x+1)^{2H-1} dx + \int_0^\infty\big( (x+1)^{\mu} - x^{\mu}\big)^2dx\right)\\
&\ = (t-s)^{2H},
\end{align*}
as required. We will omit the proof of second formula for $K_H^{MA}$, an interested reader may refer to Appendix in \cite{mishura}.
\end{proof}

Let us now turn to the harmonizable representation of fBm. 
\begin{theorem}
Let for $H\in(0,1)$ 
$$
k_t^{Ha}(x) = K^{Ha}_H \abs{x}^{-H-1/2} \begin{cases}
\sin tx, &x\ge 0, \\ 
1- \cos tx, &x<0,
\end{cases}
$$
where 
\begin{equation*}
K^{Ha}_H = \left(2\int_0^\infty \frac{1-\cos x}{x^{2H+1} dx}\right)^{-1/2} = \frac{\left(2\Gamma(2H+1)\sin \pi H\right)^{1/2}}{\pi}.
\end{equation*}

Then the process $X_t = I(k^{Ha}_t)$ is an fBm with Hurst parameter $H$.
\end{theorem}
\begin{proof}
As in the previous proof, write
\begin{align*}
&\ex{(X_t-X_s)^2}  = (K_H^{Ha})^{2}\bigg[\int_{0}^\infty \frac{\left(\sin tx - \sin sx \right)^2}{x^{2H+1}}dx \\
& \hspace{15em} + \int_{-\infty}^0 \frac{\left(\cos tx - \cos sx \right)^2}{(-x)^{2H+1}}dx \bigg]\\ 
&\qquad  = (K_H^{Ha})^{2}\int_0^\infty \frac{2 - 2\cos (t-s)x }{x^{2H+1}}dx\\
&\qquad  =  2(K_H^{Ha})^{2}(t-s)^{2H}\int_0^\infty\frac{1 - \cos z }{z^{2H+1}}dx = (t-s)^{2H}.
\end{align*}
Again, we do not proof the second formula for $K_H^{Ha}$.
\end{proof}

The third representation we consider, the so-called Volterra type representation, is a bit more involved than the former two, but its advantage is that the kernel in this representation has compact support. 
\begin{theorem}
Let for $H\in(1/2,1)$
$$
k_t^V(x) = K_H^V x^{1/2-H}\int_x^t s^{H-1/2} (s-x)^{H-3/2}ds\,\ind{[0,t]}(x),
$$
where 
$$
K_H^V = \left(\frac{H(2H-1)}{\mathrm{B}(2-2H,H-1/2)}\right)^{1/2} = K_H^{MA};
$$
for $H\in(0,1/2)$,
\begin{align*}
k_t^V(x)& = K_H^V x^{1/2-H}\bigg(t^{H-1/2}(t-x)^{H-1/2}\\
&\quad - (H-1/2)x^{1/2-H}\int_x^t s^{H-3/2} (s-x)^{H-1/2}ds\bigg)\,\ind{[0,t]}(x),
\end{align*}
where 
$$
K_H^V = \left(\frac{2H}{(1-2H)\mathrm{B}(1-2H,H+1/2)}\right)^{1/2}.
$$
Then $X_t = I(k_t^V)$ is an fBm with Hurst parameter $H$.
\end{theorem}
\begin{proof}
We will consider only the case $H\in(1/2,1)$, the other case being somewhat similar but lot more tricky.

Denote $\mu = H-1/2$ and write for $t,s\ge 0$
\begin{align*}
\ex{X_t X_s} &= (K_H^{V})^{2}\int_{0}^{t\wedge s} x^{-2\mu} 
\int_{x}^{t}u^{\mu} (u-x)^{\mu-1}du \int_{x}^{s}v^{\mu} (v-x)^{\mu -1}dv \, dx\\
& = (K_H^{V})^{2}\int_0^t \int_0^s u^{\mu}v^{\mu}\int_0^{u\wedge v} x^{-2\mu}(u-x)^{\mu-1}(v-x)^{\mu-1}dx\, dv\,du.
\end{align*}
For $u\le v$, make the change of variable $z = \frac{1-x/u}{1-x/v}$ in the inner integral so that $x = \frac{uv(1-z)}{v-zu}$, $u-x = \frac{uz(v-u)}{v-zu}$, $v-x = \frac{v(v-u)}{v-zu}$, $dx = -\frac{uv(v-u)}{(v-zu)^2}$ to obtain
\begin{align*}
&\int_0^{u} x^{-2\mu}(u-x)^{\mu-1}(v-x)^{\mu-1}dx \\
& = \int_0^1 \frac{(v-zu)^{2\mu}}{(uv)^{2\mu}(1-z)^{2\mu}}\frac{(uz)^{\mu-1}(v-u)^{\mu-1}}{(v-zu)^{\mu-1}}\frac{v^{\mu-1}(v-u)^{\mu-1}}{(v-zu)^{\mu-1}}\frac{uv(v-u)}{(v-zu)^2}dz\\
& = u^{-\mu} v^{-\mu}(v-u)^{2\mu-1}\int_0^1 z^{\mu-1}(1-z)^{-2\mu}dz\\
& = u^{-\mu} v^{-\mu} (v-u)^{2H-2} \mathrm{B}(2-2H,H-1/2),
\end{align*}
and a similar formula is valid for $v\le u$. Substituting this into the above expression for $\ex{X_tX_s}$, we arrive at
$$
\ex{X_t X_s} = H(2H-1) \int_0^t \int_0^s \abs{v-u}^{2H-2}du = \frac12 \left(t^{2H} + s^{2H} - \abs{t-s}^{2H}\right),
$$
as required.
\end{proof}

\section{Identification of Hurst parameter}\label{sec:stat}
In order to use a stochastic process as a model in practice, one needs a good statistical machinery. There are many statistical tools available for models based on the fBm, and this article is too short to cover them all. The most important statistical question is about the Hurst parameter, which governs all essential properties of fBm. 

Consider the following statistical problem: to estimate the Hurst parameter $H$ based on the observations $B_1^H, B_2^H,\dots, B_N^H$ of fBm, where $N$ is large. There are several approaches to this problem. We will study here only an approach based on discrete variations of fBm, further methods can be found in \cite{coeurjolly}. 

First we need to destroy the dependence by applying a suitable filter. Specifically, a filter of order $r$ is a polynomial $a(x) = \sum_{k=0}^{q} a_k x^k$ such that $a(1) = a'(1) =\dots = a^{(r-1)}(1) = 0$, $a^{(r)}(1)\neq 0$ (equivalently, $1$ is the root of polynomial $a$  of multiplicity $r$). The filtered observations are defined as
$$
B_n^a = \sum_{k=0}^{q} a_k B^H_{n+k},\quad n = 1,2,\dots, N-q.
$$
Popular filters are Increments 1 with $a(x) = x-1$, Daubechies 4 with $a(x) = \frac14(x-1)(x^2(1-\sqrt 3) -2x)$, Increments 2 with $a(x) = (x-1)^2$. The first two filters are of order 1, the third, of order 2. As it was mentioned, the main aim of filtering is to reduce dependence of the data. Indeed, for a filter $a$ of order $r\ge 1$ consider the covariance
\begin{align*}
\ex{B_n^a B_{m}^a}& = \sum_{k=0}^{q}\sum_{j=0}^{q} a_k a_j\ex{B^H_{n+k}B^H_{n+k}}\\
& = \frac{1}{2}\sum_{k=0}^{q} \sum_{j=0}^{q} a_k a_j \big( (n+k)^{2H}+ (m+j)^{2H} - \abs{m+k-n-j}^{2H}\big)\\
& = \frac12 \bigg(\sum_{k=0}^{q} a_k (n+k)^{2H}\sum_{j=0}^{q}a_j + \sum_{j=0}^{q} a_j (m+j)^{2H}\sum_{k=0}^{q}a_k\\
&\qquad\qquad - \sum_{k=0}^{q} \sum_{j=0}^{q}a_k a_j \abs{m+k-n-j}^{2H}\bigg)\\
& = - \frac{1}{2}\sum_{k=0}^{q} \sum_{j=0}^{q} a_k a_j \abs{m-n+k-j}^{2H}=: \rho_H^a(m-n),
\end{align*}
where we have used that $\sum_{k=0}^{q}a_k = a(1) = 0$. Consequently, the filtered data $B^a_1,\dots, B^a_{N-q}$ is a stationary process. Moreover, since $(x-1)^r\mid a(x)$, in the expression for $\rho_H^a$ one takes the finite difference of the function $x^{2H}$ \ $2r$ times: $r$ times with respect to $n$ and $r$ times with respect to $m$. It follows that 
$\rho_H^a (n)\sim K_{H,a} n^{2(H-r)}$, thus the  covariance indeed decays faster for large $r$. 

To define an estimator for the Hurst coefficient, for $m\ge 1$ consider the dilated filter $a^m(x):= a(x^m) = \sum_{k=0}^{q} a_k x^{km}$. It is obvious that $\rho_H^{a^m}(0) = m^{2H}\rho_H^{a}(0)$, equivalently,
\begin{equation}\label{log-regression}
\log \rho_H^{a^m}(0) = 2H\log m + \log\rho_H^{a}(0).
\end{equation}
Thus, an estimator for $H$ may be obtained by taking a linear regression of estimators for $\log \rho_H^{a^m}(0)$ on $\log m$. To estimate $\rho_H^{a^m}(0)$ consistently, one can use the empiric moments. 
\begin{theorem}\label{thm:consistency}
The empiric variance 
$$
V_N^{a^m} = \frac{1}{N-mq} \sum_{k=1}^{N-mq}\left(B^{a^m}_k\right)^2
$$
is a strongly consistent estimator of $r_H^{a^m}(0)$, i.e.\ 
$V_N^{a^m}\to r_H^{a^m}(0)$ a.s.\ as $N\to\infty$.
\end{theorem}
\begin{proof}
Since the sequence $\set{B_k^{a^m},k\ge 1}$ is stationary, the result follows immediately from the ergodic theorem.
\end{proof}
\begin{corollary}\label{thm:Hestimator}
Let a set $M\subset \N$ contain at least two elements, and $\widehat k^{a,M}_N$ be the coefficient of linear regression of $\set{\log V_N^{a^m},m\in M}$ on $\set{\log m, m\in M}$. Then the statistic $\widehat H^{a,M}_N = \widehat k^{a,M}_N/2$ is a strongly consistent estimator of $H$.
\end{corollary}
\begin{proof}
Follows directly from Theorem~\ref{thm:consistency} and equation \ref{log-regression}.
\end{proof}
\begin{remark}
Evidently, the same procedure can be used to estimate the Hurst parameter from observations $c B_1^H,c B^H_2,\dots, c B^H_N$ of fBm multiplied by an unknown scale coefficient $c$. This will not cause any problem, as in \eqref{log-regression} we would have an extra term $\log c$, which does not influence the estimation procedure. Moreover, thanks to the self-similarity property, the estimation procedure will not change if the scaled fBm is observed not at the positive integer points, by at the points of some other equidistant grid, i.e.\ if one observes the values $c B^H_{\Delta}, c B^H_{2\Delta}, \dots, c B^H_{N\Delta}$. It is even possible to take $\Delta =T/N$ so that we observe the values on some fixed interval. However,  one needs a different \textit{strong} consistency proof, as in this case the ergodic theorem gives only the convergence in probability.
\end{remark}
The simplest example of the regression set $M$ in Corollary~\ref{thm:Hestimator} is $M=\set{1,2}$, and the simplest example of the filter is Increments 1, $d(x) =x-1$. We get the following standard strongly consistent estimator of $H$:
$$
\widehat{H}_k =\frac{1}{2} = \frac{1}{2\log 2}(\log V^{d^2}_N-\log V^{d}_N)=\frac12\log_2 \frac{V^{d^2}_N}{V^{d}_N},
$$
where
\begin{equation*}
V^{d}_N =\frac{1}{N-1}\sum_{k=1}^{N-1}\left(B^H_{k+1}-B^H_k\right)^2,\quad
V^{d^2}_N  =\frac{1}{N-2}\sum_{k=1}^{N-1}\left(B^H_{k+2}-B^H_k\right)^2.
\end{equation*}

Let us now turn to the asymptotic normality of the coefficients. We start by formulating a rather general statement.

Let $\xi_1,\xi_2,\dots$ be a stationary sequence of standard Gaussian variables with covariance $\rho(n) =\ex{\xi_1\xi_{n+1}}$, and $g\colon \R\to \R$ be a function such that $\ex{g(\xi_1)}=0$, $\ex{g(\xi_1)^2}<\infty$. The latter assumption means that $g\in L^2(\R,\gamma)$, where $\gamma$ is the standard Gaussian measure on $\R$. Consequently, $g$ can be expanded in a series $g(x)=\sum_{k=0}^\infty g_k H_k(x)$ with respect to a system $\set{H_k,k\ge 0}$ of orthogonal polynomials for the measure $\gamma$, which are Hermite polynomials. We have $g_0 = \ex{g(\xi_1)} = 0$. The smallest number $p$ such that $g_p\neq 0$ is called the Hermite rank of $g$.

The following theorem describes the limit behavior of the cumulative sums $S_N = \sum_{k=1}^{N} g(\xi_k)$.
\begin{theorem}[Breuer--Major]
Assume that $\sum_{n=1}^{\infty} \abs{\rho(n)}^p <\infty$. Then one has the following convergence in finite-dimensional distributions:
$$
\set{\frac{1}{\sqrt N}S_{[Nt]}, t\ge 0} \to \set{\sigma_{\rho,g} W_t,t\ge 0},\quad N\to\infty,
$$
where 
$$
\sigma^2_{\rho,g}= \sum_{k=p}^\infty g_k^2 k! \sum_{n\in\mathbb Z} \rho(n)^k.
$$
\end{theorem}
As a corollary, we get asymptotic normality of the estimators. The statement depends on $r$, the order of filter $a$.
\begin{theorem}
Let either $H\in(0,3/4)$ or $H\in[3/4,1)$ and $r\ge 2$. Then for any $m\ge 1$ the estimator $V^{a^m}_N$ is an asymptotically normal estimator of $\rho^{a^m}_H(0)$, and $\widehat{H}^{a,M}_N$ is an asymptotically normal estimator of $H$.
\end{theorem}
\begin{proof}
We prove only the statement for $V^{a^m}_N$, the one for  $\widehat{H}^{a,M}_N$ follows by writing explicitly the coefficient of linear regression and analyzing asymptotic expansions. 

Write
\begin{align*}
&\sqrt{N-mq}\big(V^{a^m}_N - \rho_H^{a^m}(0)\big) = \frac{1}{\sqrt{N-mq}}\sum_{k=1}^{N-mq}\big((B^{a^m}_k)^2 - \rho_H^{a^m}(0)\big)\\
&\quad = \frac{\rho_H^{a^m}(0)}{\sqrt{N-mq}}\sum_{k=1}^{N-mq}\big(\xi_k^2 - 1\big),
\end{align*}
where $\xi_k = B^{a^m}_k/\sqrt{\rho_H^{a^m}(0)}$ is standard Gaussian. Obviously, $\rho(n):=\ex{\xi_k\xi_{k+n}} = \rho^{a^m}(n)/\rho_H^{a^m}(0)$. Thus, we are in a position to apply the Breuer--Major theorem with $g(x) = x^2 - 1$, which is obviously of Hemrite rank $2$. So we get the statement provided that $\sum_{n=1}^{\infty}\rho(n)^2<\infty$. It was argued above that $\rho^{a^m}(n)\sim K_{H,a}m^{2H} n^{2(H-r)}, n\to+\infty$. Therefore, $\sum_{n=1}^{\infty}\rho(n)^2<\infty$ iff $4(H-r)<-1$, equivalently, 
$r>H+1/4$, which is exactly our assumption.
\end{proof}
\begin{remark}
The last theorem can be used to construct approximate confidence intervals for $H$. It is possible to compute the asymptotic variance explicitly, but the expression for it is quite cumbersome, so it is not given here. A somewhat better approach is to numerically calculate it based on simulated data; the next section explains how to simulate fBm. Another observation is that the statement above depends on the value of $H$, which is a priori unknown and should be estimated. So, if one needs to construct a confidence interval for $H$,  I suggest using a filter of order $2$ unless it is \textit{a priori} known that $H<3/4$. 
\end{remark}

\section{Simulation of fractional Brownian motion}\label{sec:sim}
Among many methods to simulate fBm, the most efficient one is probably the Wood--Chan, or circulant method. The main idea is that a Gaussian vector $\xi$ with mean $\mu$ and covariance matrix $C$ can be represented as $\xi = \mu + S\zeta$, where $\zeta$ is a standard Gaussian vector, and the matrix $S$ is such that $S S^\top = C$. So in order to simulate a Gaussian vector, one needs to find a ``square root'' of covariance matrix.

Suppose that we need to simulate the values of fBm on some interval $[0,T]$. For practical purposes it is enough to simulate the values at a sufficiently fine grid, i.e.\ at the points $t_k^N = kT/N$, $k=0,1,\dots,N$ for some large $N$. Since an fBm is self-similar and has stationary increments, it is enough to simulate the values $B_1^H,B_2^H,\dots, B_N^H$ and multiply them by $(T/N)^H$. In turn, in order to simulate the latter values, it is suffices to simulate the increments $\xi_1 = B_1^H, \xi_2 = B_2^H - B_1^H, \dots, \xi_N = B_N^H - B_{N-1}^H$. The random variables $\xi_1,\xi_2,\dots,\xi_N$ form a stationary sequence of standard Gaussian variables with covariance
$$
\rho^{\vphantom{j}}_H(n)  = \ex{\xi_{1}\xi_{n+1}} = \frac{1}{2}\left((n+1)^{2H}+(n-1)^{2H} - 2n^{2H}\right),\ n\ge 1;
$$
this is so-called fractional Gaussian noise (fGn).
In other words, $\xi = (\xi_1,\dots,\xi_N)^\top$ is a centered Gaussian vector with covariance matrix
$$\operatorname{Cov}(\xi) = \begin{pmatrix}
1 &\rho^{\vphantom{j}}_H(1) &\rho^{\vphantom{j}}_H(2) &\dots & \rho^{\vphantom{j}}_H(N-2) &\rho^{\vphantom{j}}_H(N-1)\\
\rho^{\vphantom{j}}_H(1) & 1 &\rho^{\vphantom{j}}_H(1) &\dots & \rho^{\vphantom{j}}_H(N-3) &\rho^{\vphantom{j}}_H(N-2)\\
\rho^{\vphantom{j}}_H(2)& \rho^{\vphantom{j}}_H(1) & 1  &\dots & \rho^{\vphantom{j}}_H(N-4) &\rho^{\vphantom{j}}_H(N-3)\\
\vdots & \vdots & \vdots & \ddots & \vdots & \vdots \\
\rho^{\vphantom{j}}_H(N-2) & \rho^{\vphantom{j}}_H(N-3) & \rho^{\vphantom{j}}_H(N-4) & \dots & 1 & \rho^{\vphantom{j}}_H(1)\\
\rho^{\vphantom{j}}_H(N-1) & \rho^{\vphantom{j}}_H(N-2) & \rho^{\vphantom{j}}_H(N-3) & \dots & \rho^{\vphantom{j}}_H(1) & 1
\end{pmatrix}.$$

Finding a square root of $\operatorname{Cov}(\xi)$ is not an easy task. It appears that one can much easier find a square root of some bigger matrix. Specifically, put $M = 2(N-1)$ and
\begin{equation}\label{c_k}
\begin{aligned}
c_0 & = 1,\\
c_{k} & = \begin{cases}
\rho^{\vphantom{j}}_H(k), & k = 1,2,\dots, N-1,\\
\rho^{\vphantom{j}}_H(M-k), & k = N,N+1,\dots, M-1.
\end{cases}
\end{aligned}
\end{equation}
Now define a circulant matrix
$$C = \operatorname{circ}(c_0,c_1,\dots,c_{M-1})=\begin{pmatrix}
c_0 &c_1 &c_2 &\dots & c_{M-2} &c_{M-1}\\
c_{M-1} & c_0 &c_1 &\dots & c_{M-3} &c_{M-2}\\
c_{M-2}& c_{M-1} & c_0  &\dots & c_{M-4} &c_{M-3}\\
\vdots & \vdots & \vdots & \ddots & \vdots & \vdots \\
c_{2} & c_{3} & c_{4} & \dots & c_0 & c_1\\
c_{1} & c_{2} & c_{3} & \dots & c_{M-1} & c_0
\end{pmatrix}.$$  
Now define a matrix $Q = (q_{jk})_{j,k=0}^{M-1}$, with 
$$
q_{jk} = \frac{1}{\sqrt{M}}\exp\set{-2\pi i \frac{jk}{M}}.
$$
Observe that $Q$ is unitary: $Q^*Q = QQ^* = I_M$, the identity matrix. The multiplication by matrix $Q$  acts, up to the constant $1/\sqrt{M}$, as taking the discrete Fourier transform (DFT); the multiplication by $Q^*$ is, up to the constant $\sqrt{M}$, taking the inverse DFT. The following statement easily follows from the properties of DFT and its inverse.
\begin{theorem}
The circulant matrix $C$ has a representation $C = Q \Lambda Q^*$, where 
$\Lambda = \operatorname{diag}(\lambda_0,\lambda_1,\dots,\lambda_{M-1})$, $\lambda_k = \sum_{j=0}^{M-1}c_j \exp\set{-2\pi i\frac{jk}{M}}$.
Consequently, $C = S S^*$ with $S = Q \Lambda^{1/2} Q^*$, $\Lambda^{1/2} = \operatorname{diag}(\lambda^{1/2}_0,\lambda^{1/2}_1,\dots,\lambda^{1/2}_{M-1})$.
\end{theorem}
The only problem with the last statement is that, generally speaking, the matrix $S$ is complex. However, in the case of fBm the matrix $C$ is positive definite, so all the eigenvalues $\lambda_k$ are positive, as a result, the matrix $S$ is real. Thus, in order to simulate the fGn, one needs to simulate a vector $(\zeta_1,\zeta_2,\dots,\zeta_M)^\top$ of standard Gaussian variables, multiply it by $S$ and take the first $N$ coordinates of the resulting vector.

Let us turn to the practical realization of the algorithm. We start by noting that it is enough to compute the matrix $S$ only once, then one can simulate as many realizations of fGn as needed. However, I do not recommend to proceed this way. It is usually better to compute the product $Q\Lambda^{1/2}Q^*\zeta$ step by step. First compute $\frac{1}{\sqrt{M}}Q^*\zeta$, taking the inverse DFT of $\zeta$. Then multiply the result by $\Lambda^{1/2}$, i.e.\ multiply it elementwise by the vector $(\lambda^{1/2}_0,\lambda^{1/2}_1,\dots,\lambda^{1/2}_{M-1})^\top$. The last step is the multiplication by $\sqrt{M}Q$, which is made by taking the DFT. As a result, we have one DFT computation, one elementwise multiplication, and one inverse DFT computation, which are usually faster than a single matrix multiplication. 

Now it is a good moment to explain what is meant by ``usually'' in the last paragraph. It is well known that the DFT computation is most efficient when the size of data is a power of $2$; it is made by the so-called fast Fourier transform (FFT) algorithm. So, if one need to simulate e.g.\ $N=1500$ values of fGn (so that $M=2998$), it will be better (and faster) to simulate $2049$ values (so that $M=4096=2^{12}$).

Finally, taking in account everything said, we describe the algorithm.

\begin{enumerate}[1.]
\item Set $N = 2^q + 1$ and $M = 2^{q+1}$.
\item Calculate $\rho^{\vphantom{j}}_H(1),\dots,\rho^{\vphantom{j}}_H(N-1)$ and set $c_0,c_1,\dots, c_{M-1}$ according to \eqref{c_k}.
\item Take FFT to get $\lambda_0,\dots, \lambda_{M-1}$. Theoretically, one should get real numbers. However, since all computer calculations are imprecise, the resulting values will have tiny imaginary parts, so one needs to take the real part of result. 
\item Generate independent standard Gaussian $\zeta_1,\dots,\zeta_M$
\item Take the real part of inverse FFT of $\zeta_1,\dots,\zeta_M$ to obtain  $\frac{1}{\sqrt{M}} Q^* (\zeta_1,\dots,\zeta_M)^\top$. 
\item Multiply the last elementwise by $\sqrt{\lambda_0},\sqrt{\lambda_1},\dots, \sqrt{\lambda_{M-1}}$
\item Take FFT of result to get 
$$
(\xi_1,\dots,\xi_M)^\top = \sqrt{M}Q \Lambda^{1/2}\frac1{\sqrt{M}} Q^* (\zeta_1,\dots,\zeta_M)^\top = S (\zeta_1,\dots,\zeta_M)^\top.
$$
\item Take the real part of $\xi_1,\dots,\xi_N$ to get the fractional Gaussian noise.
\item Multiply by $(T/N)^H$ to obtain the increments of fBm.
\item Take cumulative sums to get the values of fBm.
\end{enumerate}
For reader's convenience I give a Matlab code of (steps 1--8 of) this algorithm. It is split into two parts: the computation of $\Lambda^{1/2}$, which can be done only once, and the simulation. 

\lstset{language=matlab,xleftmargin=.5cm,basicstyle=\ttfamily,keywordstyle=\bf,morekeywords={bsxfun}}
\begin{lstlisting}
function res = Lambda(H,N) 
M = 2*N - 2;
C = zeros(1,M);
G = 2*H;
fbc = @(n)((n+1).^G + abs(n-1).^G - 2*n.^G)/2; 
C(1:N) = fbc(0:(N-1)); 
C(N+1:M) = fliplr(C(2:(N-1))); 
res = real(fft(C)).^0.5;

function res = FGN(lambda,NT) 
if (~exist('NT','var'))
        NT = 1;
end
M = size(lambda,2);
a = bsxfun(@times,ifft(randn(NT,M),[],2),lambda);
res = real(fft(a,[],2));
res = res(:,1:(M/2));
\end{lstlisting}
To simulate $n$ realizations of fGn, use the following code. Note that for large values of $N$ and $n$, due to possible memory issues, it may be better to simulate the realizations one by one, using \texttt{FGN(lambda,1)} or simply \texttt{FGN(lambda)}. 
\begin{lstlisting}
H = 0.7; q = 10; % or whatever you like
N = 2^q + 1;
lambda = Lambda(H,N);
n = 20; % or whatever you like
fGnsamples = FGN(lambda,20);
\end{lstlisting}

\bibliographystyle{ws-procs9x6}
\bibliography{fbm-ln}

\end{document}